\documentclass[11pt,fleqn,a4paper]{scrartcl}

\usepackage[a4paper,left=4cm,right=4cm,top=3cm,bottom=3cm]{geometry}

\usepackage{verbatim}
\usepackage[english]{babel}
\usepackage[latin1]{inputenc} 
\usepackage{amssymb,amsmath,amsthm,amsfonts}
\usepackage{mathrsfs}

\usepackage{xcolor}

\usepackage[T1]{fontenc} 

\usepackage{url}        
\usepackage{enumerate}

\usepackage[all]{xy}

\usepackage{makeidx}    
\usepackage{showidx}

\usepackage{fancyhdr}

\renewcommand{\subsectionmark}[1]{} 
\fancypagestyle{plain}{
	\fancyhead{} 
	 
}

\theoremstyle{plain}

\theoremstyle{plain}
\newtheorem{theorem}{Theorem}[section]
\newtheorem{proposition}[theorem]{Proposition}
\newtheorem{lemma}[theorem]{Lemma}

\theoremstyle{definition}
\newtheorem{definition}[theorem]{Definition}

\theoremstyle{remark}

\newcommand{\N}{\mathbb N}

\newcommand{\R}{\mathbb R} 
\newcommand{\C}{\mathbb C}

\newcommand{\K}{\mathbb K}

\renewcommand{\epsilon}{\varepsilon}

\newcommand{\set}[2]{ \left\{  #1  :  #2  \right\} }

\newcommand{\smset}[1]{\{ #1 \}}

\newcommand{\func}[5]{#1 \colon #2 \longrightarrow #3 : #4 \mapsto #5}
\newcommand{\smfunc}[3]{#1 \colon #2 \longrightarrow #3}
\newcommand{\nnfunc}[4]{#1 \longrightarrow #2 : #3 \mapsto #4}
\newcommand{\nnsmfunc}[2]{#1 \longrightarrow #2}
\newcommand{\smnnfunc}[2]{\nnsmfunc{#1}{#2}}

\newcommand{\Func}[5]{
\begin{array}{rl}
   #1 : #2 & \longrightarrow #3   \\
        #4 & \longmapsto     #5   \\
\end{array}
}

\newcommand{\nnFunc}[4]{
\begin{array}{rl}
        #1 & \longrightarrow #2   \\
        #3 & \longmapsto     #4   \\
\end{array}
}

\newcommand{\oBallin}[3]{\mathrm{B}_{#1}^{#2}\left(#3 \right)}

\newcommand{\norm}[1]{\ensuremath{\left\| #1 \right\|}}

\newcommand{\opnorm}[1]{\ensuremath{  \norm{  #1 }_{\mathrm{op}} }\!  }
\newcommand{\infnorm}[1]{\ensuremath{  \norm{  #1 }_\infty }}
\newcommand{\supnorm}[1]{\infnorm{#1}}

\newcommand{\Xnorm}[1]{\ensuremath{  \norm{  #1 }_{X} }}

\newcommand{\Znorm}[1]{\ensuremath{  \norm{  #1 }_{Z} }}

\newcommand{\gnorm}[1]{\ensuremath{  \norm{  #1 }_{\g} }}

\newcommand{\abs}[1]{\ensuremath{| #1 |}}

\newcommand{\g}{\mathfrak{g}}
\renewcommand{\L}{\mathbf{L}}

\newcommand{\Exp}{\mathrm{Exp}}

\newcommand{\generatedby}[1]{\ensuremath{ \left\langle   #1  \right\rangle  }}

\newcommand{\seqn}[1]{\left(#1\right)_{n\in \N}}

\newcommand{\BoundOp}[1]{\mathcal{L}\left(#1\right)}

\newcommand{\BoundOpFromTo}[2]{\mathcal{L}\left(#1,#2\right)}

\newcommand{\BC}[2]{\mathrm{BC} \left( #1 , #2 \right)}

\newcommand{\injepsilon}{\epsilon_\circ}

\newcommand{\Sym}[3]{\mathrm{Sym}^{#1} \left( #2 , #3 \right)}
\newcommand{\Pow}[3]{\mathrm{Pow}^{#1} \left( #2 , #3 \right)}
\newcommand{\Pol}[3]{\mathrm{Pol}^{#1} \left( #2 , #3 \right)}

\newcommand{\Hold}[4]{\mathrm{BC}^{#1,#2}\!\!\left( #3, #4 \right) }
\newcommand{\Hnorm}[3]{\ensuremath{  \norm{  #1 }_{(#2,#3)} }}
\newcommand{\HSnorm}[3]{\ensuremath{  p_{(#2,#3)}(#1) }}
\newcommand{\diam}{\mathrm{diam}}
\newcommand{\FreC}{\mathrm{FC}}

\newcommand{\CC}{\mathrm{C}}

\newcommand{\BCH}{\emph{BCH}}

\newcommand{\Frechet}{Fréchet}

\newcommand{\Holder}{Hölder}

\usepackage{cite}

\title{Lie Groups Associated to \Holder-Continuous Functions}
\author{Rafael Dahmen}
\date{August 26, 2009}

\begin{document}

\maketitle 		

\section*{Abstract}
\markright{\textsc{Abstract}}

We proof some basic tools about spaces of \Holder-continuous functions between (in general infinite dimensional) Banach spaces
and use them to construct new examples of infinite dimensional (LB)-Lie groups, following the strategy of \cite{MeinArtikel}.


\tableofcontents


\section*{Introduction}
\markright{\textsc{Introduction}}
\addcontentsline{toc}{section}{Introduction}


In \cite{MeinArtikel} (Theorem C) I gave a sufficient criterion for the union of an ascending sequence of Banach-Lie groups to be an (LB)-Lie group. The purpose of this paper is to give an example of such an ascending sequence using Banach spaces of \Holder-continuous functions.
In Section \ref{sec_DIFF} we start by stating some facts about differential calculus in infinite dimensional spaces.
In Section \ref{sec_PRELIMINARIES} we will define the concept of \Holder-continuous functions between Banach spaces and we will introduce the spaces $\Hold{k}{s}{\Omega}{Z}$ and show some properties of them.
This will be used in Section \ref{sec_LIEGROUPS} to construct Banach-Lie groups associated to these spaces. Finally, we will be able to use Theorem C of \cite{MeinArtikel} to construct  (LB)-Lie groups.


\section{\Frechet-Differentiable Functions}									\label{sec_DIFF}
Let $\K\in\smset{\R,\C}$.

\subsection{Definition and easy Results}
We begin with two different notions of differentiability in infinite dimensional vector spaces: (Details can be found in \cite{MR1911979} and in \cite{MR830252})
\begin{definition}[$\CC^k$ in the sense of Michal-Bastiani]							\label{def_Ck_MB}
 Let $X$ and $Z$ be locally convex topological $\K$-vector spaces and let $\Omega$ be an open nonempty subset of $X$.
 \begin{itemize}
 \item [(i)]	 
	A mapping $\smfunc{\gamma}{\Omega}{Z}$ is called $\CC^1$, if
 	for each $(x,v)\in\Omega\times X$ the directional derivative
	\[
 	 d\gamma(x,v):=\lim_{t\to0}\frac{\gamma(x+tv)-\gamma(x)}{t}
	\]
 	exists and if the map
	\[
	 \smfunc{d\gamma}{\Omega\times X}{Z}
	\]
	is continuous.
 \item [(ii)]
	Inductively, we say that $\smfunc{\gamma}{\Omega}{Z}$ is of class $\CC^k$ if it is $\CC^1$ and if $\smfunc{d\gamma}{\Omega\times X}{Z}$ is $\CC^{k-1}$. We call $\gamma$  \emph{smooth} or $\CC^\infty$ if is $\CC^k$ for all $k\in\N$.
\end{itemize}

\end{definition}

It is an easy consequence of this definition that if $\gamma$ is $\CC^1$ and $x\in\Omega$, then the following is a continuous linear map:
\[
 \func{\gamma'(x):=d\gamma(x,\cdot)}{X}{Z}{v}{d\gamma(x,v)}.
\]

The following definition of differentiability is more well-known but has the disadvantage that it only works in normed spaces:

\begin{definition}[$\FreC^k$-maps]
 Let $X$ and $Z$ be normed spaces over $\K$ and let $\Omega$ be an open subset of $X$.
\begin{itemize}
 \item [(i)] A mapping $\smfunc{\gamma}{\Omega}{Z}$ is called \emph{\Frechet-differentiable} at the point $x\in X$ if there exists a $T\in\BoundOp{X,Z}$ such that
\[
 \lim_{v\to0}\frac{\gamma(x+tv)-\gamma(x)-Tv}{\Xnorm{v}}=0
\]
 (in this case, this map $T$ is equal to $\gamma'(x) = d\gamma(x,\cdot) $ as defined in Definition \ref{def_Ck_MB}).
 \item [(ii)] The map $\gamma$ is called $\FreC^1$ if it is everywhere \Frechet-differentiable and the map
\[
 \func{\gamma'}{\Omega}{\left(\BoundOp{X,Z},\opnorm{\cdot}\right)}{x}{\gamma'(x)=d\gamma(x,\cdot)}
\]
 is continuous.
 \item [(iii)] Inductively, we say that $\smfunc{\gamma}{\Omega}{Z}$ is of class $\FreC^k$ if it is $\FreC^1$ and if $\smfunc{\gamma'}{\Omega}{\BoundOpFromTo{X}{Z}}$ is $\CC^{k-1}$. We will use the notation $\gamma^{(1)}:=\gamma'$ and 
\[
 \gamma^{(k)}(x)(v_1,\ldots,v_k) := \left(\gamma^{(k-1)} \right)'(x) (v_1)(v_2,\ldots,v_k).
\]
 Note that each
 $\smfunc{\gamma^{(k)}(x)}{X^k}{Z}$ is a symmetric $k$-linear map.
\end{itemize}
\end{definition}

These two notions are connected via the following
\begin{lemma}[Criterion of \Frechet-Differentiability]			\label{lem_suff_crit_frec}
Let $X,Z$ be normed spaces over $\K\in\smset{\R,\C}$, $\Omega\subseteq X$ open. Then $\smfunc{\gamma}{\Omega}{Z}$ is $\FreC^1$ if and only if it is $\CC^1$ and  the map
\[
 \Func{\gamma'}{\Omega}{\left(\BoundOp{X,Z} ,\opnorm{\cdot} \right)              }
 	{x}{\left(v\mapsto \lim\limits_{t\to 0}\frac{\gamma(x+tv)-\gamma(x) }{t} \right)}
\]
 is continuous.
\end{lemma}
\begin{proof}
 If $\gamma$ is $\FreC^1$, it is clearly $\CC^1$ and $\gamma'$ is continuous. Conversely, we assume that $\gamma$ is $\CC^1$ and that $\gamma'$ is continuous. We will show that $\gamma$ is \Frechet-differentiable at each point.
 Therefore, let $x\in \Omega$ be fixed and let $v$ so small that the interval $[x,x+v]:=\set{x+tv}{t\in[0,1]}$ lies in $\Omega$. Then we define the curve
\[
 \func{\eta_v}{[0,1]}{Z}{t}{\gamma(x+tv).}
\]
Since $\gamma$ is $\CC^1$, the curve $\eta_v$ is also $\CC^1$ with
\[
 \eta_v'(t) = d\gamma(x+tv , v) = \gamma'(x+tv).v.
\]
Now, we can write:
\begin{align*}
 \Znorm{\frac{\gamma(x+v)-\gamma(x)-\gamma'(x).v}{\Xnorm{v}}}\!\!	  &	=	\frac{1}{\Xnorm{v}\!\!\!}\Znorm{\eta_v(1)-\eta_v(0)-\gamma'(x).v}
									\\&	=	\frac{1}{\Xnorm{v}\!\!\!}\Znorm{\int_0^1 \eta_v'(t)\ dt-\gamma'(x).v}
									\\&	=	\frac{1}{\Xnorm{v}\!\!\!}\Znorm{\int_0^1 \left(\gamma'(x+tv).v-\gamma'(x).v\right) dt}
									\\&	=	\frac{1}{\Xnorm{v}\!\!\!}\int_0^1\Znorm{ \left(\gamma'(x+tv)-\gamma'(x)\right).v}\ dt
									\\&	\leq	\int_0^1\opnorm{ \gamma'(x+tv)-\gamma'(x)}\ dt
\end{align*}
The map $\smfunc{\gamma'}{\Omega}{\BoundOpFromTo{X}{Z}}$ is continuous by assumption. Therefore, the integrand on the right hand side of this inequality is continuous in $t$ and in $v$. So, the theorem of parameter dependend integrals yields that the integral tends to $0$, when $v$ converges to $0$. This concludes the proof.
\end{proof}

\subsection{Polynomials}

\begin{proposition}[Interpolation of Polynomials]					\label{prop_interpol_pol}
\newcommand{\XBall}{\oBallin{1}{X}{0}}
\newcommand{\KBall}{\oBallin{1}{\K}{0}}
Let $X$ and $Z$ be normed spaces over $\K$ and let $k\in\N_0$ be given.
 
 Denote by $\Pow{j}{\XBall}{Z}$ the vector space of all $j$-homogeneous polynomials from $X$ to $Z$, restricted to $\XBall$ regarded as a subspace of the normed space
\[
 \left( \BC{\XBall}{Z},\supnorm{\cdot}\right).
\]
 Denote by $\Pol{k}{\XBall}{Z}$ the vector space of polynomials of maximal degree $k$, which is generated by $\left(\Pow{j}{\XBall}{Z}\right)_{j\leq k}$.

 Then the map
 \[
  \nnFunc{	\prod\limits_{j=0}^k \left(\Pow{j}{\XBall}{Z},\supnorm{\cdot}\right) 	}{	\left(	\Pol{k}{\XBall}{Z}	,\supnorm{\cdot}\right)	}
		{(\gamma_j)_{j} }{	\sum_{j=0}^k \gamma_j                      }
 \]
 is a topological isomorphism.
\end{proposition}
\begin{proof}
\newcommand{\XBall}{\oBallin{1}{X}{0}}
\newcommand{\KBall}{\oBallin{1}{\K}{0}}
The map is clearly bijective and continuous. It remains to show that for every $j_0\leq k$ the coefficient map
\[
 \nnFunc	{	\left(	\Pol{k}{\XBall}{Z}	,\supnorm{\cdot}\right)	}{	\left(\Pow{j_0}{\XBall}{Z},\supnorm{\cdot}	\right)	}
		{\gamma=\sum_{j=0}^k \gamma_j }{\gamma_{j_0}}
\]
is continuous.

We fix a subset $F\subseteq ]0,1[$ with $k+1$ elements. For every point $\mu\in F$ we define the corresponding Lagrange polynomial:
\[
	\Lambda_\mu(t) :=\prod_{\substack{\nu\in F\\ \nu\neq\mu }}\frac{t-\nu}{\mu-\nu}
			= \sum_{j=0}^k \lambda_{\mu,j}\  t^j \in \R[t]
\]
This is the unique polynomial of degree $k$ such that $\Lambda_\mu(\nu)=\delta_{\mu,\nu}$ for $\nu\in F$.
The coefficients $\lambda_{\mu,j}\in\R$ depend only on $k$ and $F$ and are therefore considered fixed for the rest of the proof.

Now, suppose that a function $\smfunc{g}{F}{Z}$ from the finite set $F$ into the normed space $Z$ is given. Then there is a unique polynomial $\smfunc{\widetilde{g}}{\K}{Z}$ such that $\widetilde{g}|_F = g$. This polynomial is given by:
\[
 \widetilde{g}(t) 	:= \sum_{\mu\in F} g(\mu)\cdot \Lambda_\mu(t)
		 = \sum_{j=0}^k \left(\sum_{\mu\in F} g(\mu)\cdot  \lambda_{\mu,j}\right)  t^j
\]
We may estimate the norm of the $j$-th coefficient of $\widetilde{g}$:
\begin{align*}
 \Znorm{\sum_{\mu\in F} g(\mu)\cdot  \lambda_{\mu,j}}	  &	\leq	\sum_{\mu\in F} \abs{\lambda_{\mu,j}} \supnorm{ g}	  .
\end{align*}

Now, we consider a continuous polynomial $\smfunc{\gamma= \sum_{j=0}^k \gamma_k	}{X}{Z}$, where each $\gamma_j$ is a continuous $j$-homogeneous polynomial. Let $v\in \XBall$. Then $\gamma_{j_0}(v)$ is the $j_0$-th coefficient of the polynomial 
\[
 g_v(t) := \gamma(tv) = \sum_{j=0}^k \gamma_j(v)\  t^j
\]
and we may estimate its norm by:
\begin{align*}
 \Znorm{\gamma_{j_0}(v)}	  &	\leq	\sum_{\mu\in F} \abs{\lambda_{\mu,j}} \supnorm{g_v}	  
				\leq	\sum_{\mu\in F} \abs{\lambda_{\mu,j}} \supnorm{g|_{\XBall}}	  .
\end{align*}
Since $v\in \XBall$ was arbitrary, this shows
\[
 \supnorm{\gamma_j|_{\XBall}}\leq \left(\sum_{\mu\in F} \abs{\lambda_{\mu,j}}\right) \supnorm{\gamma|_{\XBall}}
\]
which finishes the proof.
\end{proof}

\begin{proposition}[Taylor's Formula]							\label{prop_taylor}
 Let $X$ and $Z$ be normed spaces over $\K$ and let $\Omega$ be an open convex subset of $X$ and $x_0\in X$. Assume $\smfunc{\gamma}{\Omega}{Z}$ is $\FreC^{k}$ with $k\geq 1$. Then we have for all $v\in X$ such that $x+v\in\Omega$:
\begin{align*}
 \hbox{(a)\quad}	\gamma(x_0+v)			=& \!  \sum_{j\leq k-1} \frac{	\gamma^{(j)}(x_0)(v,\ldots,v)	}{j!}				 \\
							 & + \int_0^1\frac{(1-t)^{k-1}}{(k-1)!} \gamma^{(k)}(x_0+tv) (v,\ldots,v) dt. 	\\
 \hbox{(b)\quad}	\gamma(x_0+v)			=&   \sum_{j\leq k} \frac{	\gamma^{(j)}(x_0)(v,\ldots,v)	}{j!}				 \\
							 & + \int_0^1 \frac{(1-t)^{k-1}}{(k-1)!}\! \left(  \gamma^{(k)}(x_0+tv) -\gamma^{(k)}(x_0)\right)\!\! (v,\ldots,v)  dt. 
\end{align*}
\end{proposition}
\vspace{-0.5cm}
\begin{proof}
 By setting $\func{h}{]-r,r[}{Z}{s}{\gamma(x_0+sv)}$ and using continuous linear functionals on $F$, we can reduce (a) to the classical formula where $X$ and $Z$ are one-dimensional.

 If we split the difference $\left(  \gamma^{(k)}(x_0+tv) -\gamma^{(k)}(x_0)\right)$ in the integral on the right hand side of (b) into two integrals and simplify the expression, it is easy to see that (b) follows from (a).
\end{proof}

\section{Spaces of \Holder-Continuous Functions}										\label{sec_PRELIMINARIES}
Throughout this section, let $\Omega$ be a convex bounded
open
subset of a \emph{real} Banach space $X$.
\begin{definition}[\Holder-Spaces]
Let $Z$ be a Banach space over the field \hbox{$\K\in\smset{\R,\C}$.}
\begin{itemize}
   \item [(a)] We set $\Hold{0}{0}{\Omega}{Z} := \BC{\Omega}{Z}$ to be vector space of all bounded continuous $Z$-valued functions on the set $\Omega$. It will always be endowed with the norm  $\Hnorm{\cdot}{0}{0}:=\HSnorm{\cdot}{0}{0}:=\supnorm{\cdot}$.
  \item [(b)] For a real number $s\in]0,1]$, we set
	\[
	 \Hold{0}{s}{\Omega}{Z} := \set{ \smfunc{\gamma}{\Omega}{Z}  }{\HSnorm{\gamma}{0}{s} := \sup_{\substack{x,y\in \Omega \\ x\neq y}}	\frac{\Znorm{\gamma(x)-\gamma(y)} }{\Xnorm{x-y}^s}<\infty}.
	\]
	From this definition follows at once that every $\gamma\in \Hold{0}{s}{\Omega}{Z}$ is uniformly continuous and bounded.
	We endow this vector space with the norm $\Hnorm{\cdot}{0}{s}:= \supnorm{\cdot} + \HSnorm{\cdot}{0}{s}$.
  \item [(c)] Recursively, we may define 
	\[
	 \Hold{k+1}{s}{\Omega}{Z} := \set{ \gamma\in \FreC^1(\Omega,Z) }{ \gamma' \in \Hold{k}{s}{\Omega}{ \BoundOpFromTo{X}{Z} } }
	\]
	for $k\in\N_0$ and $s\in[0,1]$. We endow this vector space with the norm $\Hnorm{\cdot}{k+1}{s}:=\supnorm{\cdot} + \HSnorm{\cdot}{k+1}{s}$ which is defined as
	\[
	 \HSnorm{\gamma}{k+1}{s} := \HSnorm{\gamma'}{k}{s}.
	\]
 \end{itemize}
\end{definition}

\subsection{Inclusion Mappings}
In this subsection, we will show that the inclusion operators between these spaces are continuous (Proposition \ref{prop_incl_general}).

We begin with the following special case where the inclusion operator behaves very nicely:
\begin{proposition}										\label{prop_incl_embed}
 For every $k\in\N_0$ the vector space $\Hold{k+1}{0}{\Omega}{Z}$ is a vector subspace of $\Hold{k}{1}{\Omega}{Z}$ and the inclusion map is an isometric embedding.
\end{proposition}
\begin{proof}
 Since for $(k,s)\neq(0,0)$ the norm $\Hnorm{\cdot}{k}{s}$ is the sum of the $\supnorm{\cdot}$-norm and the $\HSnorm{\cdot}{k}{s}$-seminorm, it suffices to show that for every $\gamma\in \Hold{k+1}{0}{\Omega}{Z}$ the seminorms are equal:
\[
	\HSnorm{\gamma}{k}{1} = \HSnorm{\gamma}{k+1}{0}.
\]
 It suffices to show this for $k=0$. The rest follows immediately by induction on $k$.
 Let $\gamma\in \Hold{1}{0}{\Omega}{Z}$ be given. By definition of the \Holder-spaces, this means $\gamma$ is continuously differentiable with bounded \Frechet-derivative. Now, we estimate
\begin{align*}
 \Znorm{\gamma(x)-\gamma(y)} 	  &= 		\Znorm{	\int_0^1 \gamma'\big(tx+(1-t)y\big)\big(x-y\big) dt	}	
				\\&\leq		\supnorm{\gamma'}\Xnorm{x-y}
				\\&=   		\HSnorm{\gamma}{1}{0}\Xnorm{x-y}
\end{align*}
 This yields:
\[
 \HSnorm{\gamma}{0}{1}\leq \HSnorm{\gamma}{1}{0}.
\]
But conversely: Let $x_0\in \Omega, v\in X$ with $\Znorm{v}=1$ and $t\in\R^\times$ (small enough) be given. Then we may estimate:
\begin{align*}
 \Znorm{\frac{1}{t}\big( \gamma(x+tv)-\gamma(x)	\big)}	  &	= 	\frac{1}{\abs{t}} \Znorm{\gamma(x+tv)-\gamma(x)}	
							\\&	\leq	\frac{1}{\abs{t}} \cdot \HSnorm{\gamma}{0}{1} \Znorm{(x+tv)-x}	
							\\&	=   	\HSnorm{\gamma}{0}{1}.
\end{align*}
Now, as $t$ tends to zero, the left hand side converges to $\gamma'(x).v$. Since $v$ was arbitrary with norm $1$, this yields $\opnorm{\gamma'(x)}\leq \HSnorm{\gamma}{0}{1}$ and since $x$ was arbitrary, we finally obtain:
\[
 \HSnorm{\gamma}{1}{0}\leq \HSnorm{\gamma}{0}{1}.
\]
Therefore the seminorms are equal and this finishes the proof.
\end{proof}

\begin{proposition}								\label{prop_incl_s}
 Let $k\in\N_0$ and let $0<s_1<s_2\leq 1$. Then the vector space  $\Hold{k}{s_2}{\Omega}{Z}$ is a vector subspace of $\Hold{k}{s_1}{\Omega}{Z}$ and the inclusion map is continuous with operator norm at most $\max\{1,\left(\diam\Omega\right)^{s_2-s_1} \}$.
\end{proposition}

\begin{proof}
 Once again, it suffices to show this for $k=0$. 
 \begin{align*}
  \frac{\Znorm{\gamma(x)-\gamma(y)} }{ \Xnorm{x-y}^{s_1} }	  &	=	\frac{\Znorm{\gamma(x)-\gamma(y)} }{ \Xnorm{x-y}^{s_2} }\cdot \Xnorm{x-y}^{s_2-s_1}
								\\&	\leq	\HSnorm{\gamma}{0}{s_2}\cdot \left(\diam\Omega\right)^{s_2-s_1}.
 \end{align*}
 This shows 
\[
	\HSnorm{\cdot}{0}{s_1} \leq \left(\diam\Omega\right)^{s_2-s_1}\cdot \HSnorm{\cdot}{0}{s_2}.
\]
 The corresponding inequality for $\Hnorm{\cdot}{0}{s_2}$ and $\Hnorm{\cdot}{0}{s_1}$ follows immediately.
\end{proof}

\begin{lemma}													\label{lem_incl_s0}
 Let $(k,s)\in\N_0\times]0,1]$ and $x_0\in\Omega$ be fixed.
\begin{itemize}
 \item [(a)] The linear operator
		\[
		 \nnFunc{\Hold{k}{s}{\Omega}{Z}	}{\left(\Sym{k}{X}{Z},\opnorm{\cdot}\right)}
		 {\gamma}{\gamma^{(k)}(x_0)}
		\]
		is continuous.
 \item[(b)] The linear operator
		\[
		 \nnFunc{\Hold{k}{s}{\Omega}{Z}	}{\Hold{k}{0}{\Omega}{Z}	}
		 {\gamma}{\gamma}
		\]
		is continuous.
\end{itemize}
The operator norms of these operators may be bounded by constants depending on $k, \Omega$ and $x_0$, but not on $Z$ or $s$.
\end{lemma}
\begin{proof}
 \newcommand{\en}{\epsilon_0}
 \newcommand{\TAYLOR}{T^k_{x_0}\!\!\gamma(\en v)}
 \newcommand{\REMAINDER}{R\gamma(\en v)}

 For $k=0$ both, (a) and (b) are trivial. So, we may assume $k\geq 1$.

 Before we show (a), we show how (b) follows from (a):
 \begin{align*}
  \Hnorm{\gamma}{k}{0}	  &	=	\supnorm{\gamma} + \HSnorm{\gamma}{k}{0}
			\\&	\leq	\supnorm{\gamma} + \supnorm{\gamma^{(k)}}
			\\&	=   	\supnorm{\gamma} + \sup_{x\in\Omega} \opnorm{\gamma^{(k)}(x)}
			\\&	\leq   	\supnorm{\gamma} 
					+ \sup_{x\in\Omega} 	\opnorm{\gamma^{(k)}(x) - \gamma^{(k)}(x_0)} 
					+ \opnorm{\gamma^{(k)}(x_0)}
			\\&	\leq   	\supnorm{\gamma} 
					+ \HSnorm{\gamma^{(k)}  }{0}{s}\cdot (\diam\Omega)^s
					+ \opnorm{\gamma^{(k)}(x_0)}
			\\&	\leq	\supnorm{\gamma} 
					+ (\diam\Omega)\cdot \HSnorm{\gamma  }{k}{s}
					+ \opnorm{\gamma^{(k)}(x_0)}
 \end{align*}
 The first two summands are obviously continuous with respect to $\Hnorm{\gamma}{k}{s}$ and the continuity of the third summand follows from part (a).

 Now we prove (a):
 Since $\Omega$ is open, there is a constant $\en>0$ such that $\overline{\oBallin{\en}{X}{x_0}}\subseteq \Omega$.
 Let $v\in X$ be a vector with $\Xnorm{v}\leq 1$.
 Since $\gamma\in \Hold{k}{s}{\Omega}{Z}$,
  it is in particular $\FreC^k$ and therefore we may use Taylor's formula (Proposition \ref{prop_taylor} (b) ) and obtain:
\[
 \gamma(x_0+\en v) 	= \TAYLOR + \REMAINDER									\tag{$*$}\label{eqn_Taylor}
\]
 with
\[
 \TAYLOR		= \sum_{j\leq k} \frac{	\gamma^{(j)}(x_0)(v,\ldots,v)\en^j	}{j!}
\]
and
\[
 \REMAINDER		= \int_0^1 \frac{(1-t)^{k-1}}{(k-1)!} \left(  \gamma^{(k)}(x_0+t\en v) -\gamma^{(k)}(x_0)\right) (v,\ldots,v)\en^k \ dt.
\]
First, we will look at the remainder part $\REMAINDER$:
\begin{align*}
 \Znorm{\REMAINDER}	  &	=	\!\Znorm{	\int_0^1\! \frac{(1-t)^{k-1}}{(k-1)!} \!\!\left(  \gamma^{(k)}(x_0+t\en v) -\gamma^{(k)}(x_0)\!\right)\! (v,\ldots,v)\en^k dt	}
 			\\&	\leq	\int_0^1 \frac{1}{(k-1)!} \opnorm{  \gamma^{(k)}(x_0+t\en v) -\gamma^{(k)}(x_0)	} \Xnorm{v}^k \en^k \ dt	
 			\\&	\leq	\int_0^1 \frac{1}{(k-1)!} \HSnorm{\gamma }{k}{s}\cdot \Xnorm{t\en v}^s  \en^k \ dt	
 			\\&	\leq	 \underbrace{\frac{\en^{k+1}}{(k-1)!}}_{=: C_1} \Hnorm{\gamma }{k}{s} .
\end{align*}
This shows that the remainder term is bounded above by a constant (depending only on $k, \Omega$ and $x_0$) times $\Hnorm{\gamma }{k}{s}$.

Now, we estimate the norm of the Taylor-polynomial:
\begin{align*}
 \Znorm{\TAYLOR}	  &	\stackrel{\hbox{(\ref{eqn_Taylor})} }{=}	\Znorm{\gamma(x_0+\en v) - \REMAINDER}
			\\&	\leq	\Znorm{\gamma(x_0+\en v)} + \Znorm{\REMAINDER}
			\\&	\leq	\underbrace{\supnorm{\gamma}}_{\leq \Hnorm{\gamma }{k}{s}} +  C_1\Hnorm{\gamma }{k}{s}
			\\&	\leq	C_2 \Hnorm{\gamma }{k}{s}.
\end{align*}

Since $v\in \overline{\oBallin{1}{X}{0}}$ was arbitrary, this shows that the sup norm of the Taylor polynomial on the closed unit ball is bounded by a constant times $\Hnorm{\gamma }{k}{s}$. By Proposition \ref{prop_interpol_pol} the norm of every homogeneous part is bounded above by the norm of the polynomial:
\[
 \opnorm{\frac{	\gamma^{(j)}(x_0)(\cdot)\en^j	}{j!}} \leq C_3 \Hnorm{\gamma }{k}{s}.
\]
As we saw in Proposition \ref{prop_interpol_pol}, this constant does only depend on $j$ and $k$.

In particular, we have for the case $j=k$: 
\[
 \opnorm{ \gamma^{(k)}(x_0) } \leq C_4 \Hnorm{\gamma }{k}{s}
\]
which is what we had to show.
\end{proof}

\begin{proposition}										\label{prop_incl_general}
 Let $(k,s), (l,t)\in \N_0\times [0,1]$ be given. Assume $k+s < l+t$. Then
\[
 \Hold{l}{t}{\Omega}{Z} \subseteq \Hold{k}{s}{\Omega}{Z}
\]
and the inclusion map is a continuous operator whose norm can be bounded above by a constant depending only on $l$, $X$ and $\Omega$.
\end{proposition}
\begin{proof}
 This is a immediate consequence of Proposition \ref{prop_incl_embed}, Proposition \ref{prop_incl_s} and Lemma \ref{lem_incl_s0} (b).
\end{proof}

\subsection{Completeness of the \Holder-Spaces}
\begin{lemma}											\label{lem_Banach}
\newcommand{\Y}{\BoundOpFromTo{X}{Z}}
Let $s\in[0,1]$ and $k\in\N_0$ be given. Then the map
\[
 \Func{\kappa}{\Hold{k+1}{s}{\Omega}{Z}  }{	\Hold{0}{0}{\Omega}{Z} \times \Hold{k}{s}{\Omega}{\Y}	}{\gamma}{(\gamma,\gamma')}
\]
is a topological embedding.
\end{lemma}
\begin{proof}
 The map $\kappa$ is clearly linear and injective. We show the continuity of $\kappa$ with the following estimate:
 \begin{align*}
  \norm{\kappa(\gamma)}	  &	=	\supnorm{\gamma} + \Hnorm{\gamma'}{k}{s}
			   	=	\supnorm{\gamma} + \HSnorm{\gamma'}{k}{s}
					+ \supnorm{\gamma'}
			\\&	\leq	\Hnorm{\gamma}{k+1}{s} + \Hnorm{\gamma}{1}{0}
 \end{align*}
 By Proposition \ref{prop_incl_general}, $\Hnorm{\cdot}{1}{0}$ is continuous with respect to $\Hnorm{\cdot}{k+1}{s}$. This implies the continuity of $\kappa$.
 
 On the other hand, $\Hnorm{\gamma}{k+1}{s} = \supnorm{\gamma}+\HSnorm{\gamma'}{k}{s} \leq \supnorm{\gamma}+\Hnorm{\gamma'}{k}{s} = \norm{\kappa(\gamma)}$. Hence, $\kappa$ is a topological embedding.
\end{proof}

\begin{proposition}										\label{prop_Banach}
 Let $s\in[0,1]$ and $k\in\N_0$ be given. Then the normed space $\left(\Hold{k}{s}{\Omega}{Z},\Hnorm{\cdot}{k}{s}\right)$ is complete, hence a Banach space.
\end{proposition}
\begin{proof}
 For $(k,s)=(0,0)$, this is well known.
 Therefore, let $k=0$ and $s\in]0,1]$.
 \newcommand{\Rest}{U_\Omega}
 \newcommand{\Y}{\BoundOpFromTo{X}{Z}}
 For every $\gamma\in \BC{\Omega}{Z}$, define
\[
 \func{R\gamma}{\Rest}{Z}{(x,y)}{\frac{\gamma(x)-\gamma(y)}{\Xnorm{x-y}^s} }
\]
 Here, $\Rest:=\set{(x,y)\in\Omega\times\Omega}{x\neq y}$ denotes the complement of the diagonal in $\Omega\times \Omega$.
 
 Now, it is clear that $\Hold{0}{s}{\Omega}{Z}:= \set{\gamma\in \BC{\Omega}{Z}	}{R\gamma\in \BC{\Rest}{Z}	}$ and that
\[
 \Func{\iota}{	\Hold{0}{s}{\Omega}{Z}	}{	\BC{\Omega}{Z}\times \BC{\Rest}{Z}	}{	\gamma	}{	(\gamma,R\gamma)	}
\]
 is an isometric embedding. Therefore it remains to show that the image of $\iota$ is closed in the product of the two Banach spaces $\BC{\Omega}{Z}\times \BC{\Rest}{Z}$.

Now, let $(\gamma, \eta)$ be in the closure of the image of $\iota$. This implies that there is a sequence $\seqn{\gamma_n}$ in the space $\Hold{0}{s}{\Omega}{Z}$ such that
$\seqn{\gamma_n}$ converges uniformly to $\gamma \in \BC{\Omega}{Z}$ and that $\seqn{R\gamma_n}$ converges uniformly to \hbox{$\eta\in \BC{\Rest}{Z}$.}
In particular, we have pointwise convergence, hence the following holds for all $(x,y)\in \Rest$:
\[
 \eta(x,y) = \lim_{n\rightarrow\infty} \frac{\gamma_n(x)-\gamma_n(y)}{\Xnorm{x-y}^s}
\]
But the right hand side converges pointwise to $\frac{\gamma(x)-\gamma(y)}{\Xnorm{x-y}^s}$ since $\seqn{\gamma_n}$ converges to $\gamma$. Therefore $\eta=R\gamma$ and therefore the image of $\iota$ is closed and $\Hold{0}{s}{\Omega}{Z}$ is a Banach space.

Now, we will show the claim for $(k+1,s)$ and by an induction argument, we may assume that it holds for $(k,s)\in \N_0\times[0,1]$. We will use the topological embedding from Lemma \ref{lem_Banach}:
\[
 \Func{\kappa}{	\Hold{k+1}{s}{\Omega}{Z}	}{	\BC{\Omega}{Z}\times \Hold{k}{s}{\Omega}{\Y }	}{	\gamma	}{	(\gamma,\gamma')	}
\]
 So, again it suffices to show that the image of $\kappa$ is closed in the Banach space $\BC{\Omega}{Z}\times \Hold{k}{s}{\Omega}{\Y }$ which by induction hypothesis is a product of two Banach spaces.

Now, let $(\gamma, \eta)$ be in the closure of the image of $\kappa$. This implies that there is a sequence $\seqn{\gamma_n}$ in the space $\Hold{k+1}{s}{\Omega}{Z}$ such that
$\seqn{\gamma_n}$ converges to $\gamma$ in $\BC{\Omega}{Z}$ and that $\seqn{\gamma_n'}$ converges to $\eta\in \Hold{k}{s}{\Omega}{\Y}$.
We have to show that $\gamma\in \Hold{k+1}{s}{\Omega}{Z}$ and that $\gamma'=\eta$.

Therefore let $x_0\in \Omega$ and $v\in X$ be given. Since $\Omega$ is convex, we can write the difference quotient of $\gamma_n$ at point $x_0\in \Omega$ in direction $v\in X$ as:
\[
 \frac{1}{t}\left( \gamma_n(x_0+tv)-\gamma_n(x_0)	\right) = \int_0^1 \gamma_n'(x_0+stv).v  ds
\]
if $\abs{t}$ is small enough. Now, we take the pointwise limit as $n\rightarrow \infty$ and obtain:
\[
 \frac{1}{t}\left( \gamma(x_0+tv)-\gamma(x_0)	\right) = \int_0^1 \eta(x_0+stv).v  ds
\]
For the convergence of the integral, we use that $\supnorm{\gamma_n'-\eta}\rightarrow 0$.

If we let now $t$ tend to $0$, then the right hand side converges to $\eta(x_0).v$.
So, we have shown that the directional derivative of $\gamma$ at point $x_0$ in direction $v$ exists and is equal to $\eta(x_0).v$.
Since $v\in X$ was arbitrary, all directional derivatives exist and we have just seen that the map
\[
 \func{d\gamma}{\Omega\times X}{Z}{(x,v)}{\eta(x).v}
\]
is continuous, therefore $\gamma$ is $\CC^1$ the Michal-Bastiani sense. But since
\[
 \gamma'(x)=d\gamma(x,\cdot)=\eta(x)
\]
and $\smfunc{\eta}{\Omega}{\Y}$ is continuous by hypothesis, we can apply Lemma \ref{lem_suff_crit_frec} and obtain that  $\gamma$ is $\FreC^1$.
Since $\gamma'=\eta\in \Hold{k}{s}{\Omega}{\Y }$, this implies that $\gamma\in \Hold{k+1}{s}{\Omega}{Z}$ which finishes the proof.
\end{proof}

\subsection{Products of \Holder-Continuous Functions}

\begin{theorem}[Products of \Holder-Continuous Functions]												\label{thm_small}
We assume that $\diam\Omega\leq 1$. Let $\smfunc{\bullet}{Z_1\times Z_2}{Z}$ be a continuous bilinear map. We define the pointwise product of two functions $\gamma_1\in \Hold{k}{s}{\Omega}{Z_1}$ and $\gamma_2\in \Hold{k}{s}{\Omega}{Z_2}$ as
\[
 \func{\gamma_1 \bullet \gamma_2}{\Omega}{Z}{x}{\gamma_1(x)\bullet\gamma_2(x)}.
\]
Then the product is again in $\Hold{k}{s}{\Omega}{Z}$ and we have the following formula:
\[
 \Hnorm{\gamma_1\bullet \gamma_2}{k}{s}\leq C_k \cdot \opnorm{\bullet}\cdot   \Hnorm{\gamma_1}{k}{s}\Hnorm{\gamma_2}{k}{s}
\]
Here, the $C_k>0$ is a constant, depending only on $k$, but not on $s$ or on $\bullet$. This will important later on.
\end{theorem}
\begin{proof}
 By replacing the continuous bilinear map $\bullet$ by its multiple $\frac{1}{\opnorm{\bullet}}\bullet$, we may assume that $\opnorm{\bullet}=1$.
 
 The claim is trivial for $(k,s)=(0,0)$. The case $k=0$ and $s\in ]0,1]$ is done in the following way:
\begin{align*}
 \Znorm{\gamma_1\bullet\gamma_2(x)\!-\! \gamma_1\bullet\gamma_2(y)}\!	  &	\leq	\!  \Znorm{\gamma_1(x)\bullet\gamma_2(x)- \gamma_1(x)\bullet\gamma_2(y)}
									\\&	\quad 	+   \Znorm{\gamma_1(x)\bullet\gamma_2(y)- \gamma_1(y)\bullet\gamma_2(y)}
									\\&	\leq	\!    \Znorm{\gamma_1(x)} \Znorm{\gamma_2(x)-\gamma_2(y)}
									\\&	\quad	+   \Znorm{\gamma_1(x)-\gamma_1(y)}\Znorm{\gamma_2(y)}
									\\&	\leq	\!    \left(\supnorm{\gamma_1}\HSnorm{\gamma_2}{0}{s}
										+	\HSnorm{\gamma_1}{0}{s}\supnorm{\gamma_2}\right)\!\Xnorm{x-y}^s.
\end{align*}
Therefore we have
\[
 \HSnorm{\gamma_1\bullet\gamma_2}{0}{s} \leq \supnorm{\gamma_1}\HSnorm{\gamma_2}{0}{s}	+	\HSnorm{\gamma_1}{0}{s}\supnorm{\gamma_2}
\]
Now we add the inequality $\supnorm{\gamma_1\bullet\gamma_2}\leq \supnorm{\gamma_1}\supnorm{\gamma_2}$ on both sides:
\[
 \Hnorm{\gamma_1\bullet\gamma_2}{0}{s} \leq 	\underbrace{  \supnorm{\gamma_1}\HSnorm{\gamma_2}{0}{s}  }_{\leq \Hnorm{\gamma_1}{0}{s}\Hnorm{\gamma_2}{0}{s}}
			+	\underbrace{  \HSnorm{\gamma_1}{0}{s}\supnorm{\gamma_2} + \supnorm{\gamma_1}\supnorm{\gamma_2}  }_{= \Hnorm{\gamma_1}{0}{s}\Hnorm{\gamma_2}{0}{s} }. 
\]
This proves the claim for $k=0$ and $s\in [0,1]$ for the constant $C_0:=2$. Now assume the claim holds for $k$. We will show it for $k+1$.

 \newcommand{\Y}[1]{\BoundOpFromTo{X}{Z_{#1}} }
Therefore, we are given $\gamma_1 \in \Hold{k+1}{s}{\Omega}{Z_1}$ and $\gamma_2 \in \Hold{k+1}{s}{\Omega}{Z_2}$. By definition, this means that $\gamma_1$ and $\gamma_2$ are $\FreC^1$ and 
\[
 \gamma_1'\in \Hold{k}{s}{\Omega}{\BoundOpFromTo{X}{Z_1}} \quad \hbox{ and } \quad \gamma_2'\in \Hold{k}{s}{\Omega}{\BoundOpFromTo{X}{Z_2}}.
\]
Now we define the following bilinear operators:
\[
 \Func{*_1}{ Z_1\times\Y{2} }{\Y{} }{(z,T)}{ \left(	x\mapsto z\bullet(Tx)	\right) }
\]
and
\[
 \Func{*_2}{ \Y{1}\times Z_2 }{\Y{} }{(T,z)}{ \left(	x\mapsto (Tx)\bullet x	\right). }
\]
It is easy to verify that $\opnorm{*_1},\opnorm{*_2}\leq 1$. Therefore, we may use the induction hypothesis and obtain that $\gamma_1  *_1  \gamma_2'$ and $\gamma_1' *_2  \gamma_2$ belong to $\Hold{k}{s}{\Omega}{\Y{}}$ and we have the following estimates:
\[
 \Hnorm{\gamma_1  *_1  \gamma_2'}{k}{s}\leq C_k\Hnorm{\gamma_1}{k}{s}\Hnorm{\gamma_2'}{k}{s}
\]
and
\[
 \Hnorm{\gamma_1' *_2  \gamma_2}{k}{s}\leq C_k\Hnorm{\gamma_1'}{k}{s}\Hnorm{\gamma_2}{k}{s}
\]
 By the product rule for \Frechet-derivatives, we know that 
\[
 (\gamma_1\bullet\gamma_2)'=\gamma_1  *_1  \gamma_2' + \gamma_1' *_2  \gamma_2.
\]
And hence $(\gamma_1\bullet\gamma_2)'\in\Hold{k}{s}{\Omega}{\Y{}}$ which implies $\gamma_1\bullet\gamma_2\in\Hold{k+1}{s}{\Omega}{Z}$.

It remains to show the norm estimate:
\begin{align*}
 \HSnorm{\gamma_1\bullet\gamma_2}{k+1}{s}	  &	=	\HSnorm{(\gamma_1\bullet\gamma_2)'}{k}{s}
 							\leq	\HSnorm{\gamma_1  *_1  \gamma_2'}{k}{s} + \HSnorm{\gamma_1' *_2  \gamma_2}{k}{s}
 						\\&	\leq	\Hnorm{\gamma_1  *_1  \gamma_2'}{k}{s} + \Hnorm{\gamma_1' *_2  \gamma_2}{k}{s}
 						\\&	\leq	C_k\left(\Hnorm{\gamma_1}{k}{s}\Hnorm{\gamma_2'}{k}{s} + \Hnorm{\gamma_1'}{k}{s} \Hnorm{\gamma_2}{k}{s}\right)
 						\\&	\leq	C_k\left(\Hnorm{\gamma_1}{k}{s}2\HSnorm{\gamma_2}{k+1}{s} + 2\HSnorm{\gamma_1}{k+1}{s} \Hnorm{\gamma_2}{k}{s}\right)
 						\\&	\leq	C_k\left(D_k\Hnorm{\gamma_1}{k+1}{s}2\Hnorm{\gamma_2}{k+1}{s} + 2\Hnorm{\gamma_1}{k+1}{s} \Hnorm{\gamma_2}{k+1}{s}\right)
 						\\&	=   	\underbrace{(2D_k+2)C_k}_{C_{k+1} := }\Hnorm{\gamma_1}{k+1}{s}\Hnorm{\gamma_2}{k+1}{s}.
\end{align*}
 Here $D_k$ is an upper bound for the norm of the inclusion $\smnnfunc{ \Hold{k+1}{s}{\Omega}{Z} }{ \Hold{k}{s}{\Omega}{Z} }$, independent of $Z$ and $s$, which exists by \ref{prop_incl_general}.
This finishes the proof.
\end{proof}

\subsection{Directed Unions of \Holder-Spaces}
From now on, we will assume that $\diam\Omega\leq 1$. By Proposition \ref{prop_incl_s}, this implies that for a fixed $k\in\N_0$ and $0<s_1<s_2\leq 1$ the inclusion map
\[
 \nnsmfunc{\Hold{k}{s_2}{\Omega}{Z} }{ \Hold{k}{s_1}{\Omega}{Z} }
\]
is continuous with operator norm at most $1$.
\begin{proposition}[Logarithmic Convexity Property for $k=0$]							\label{prop_log_conv}
$\phantom{M}$ 
\begin{itemize}
 \item [(a)] Let $0< s<u \leq 1$. Assume $\gamma\in \Hold{0}{u}{\Omega}{Z}$ and let $\lambda\in ]0,1[$. Then we have
\[
 \HSnorm{\gamma}{0}{\lambda s + (1-\lambda)u} \leq \left(\HSnorm{\gamma}{0}{s}\right)^\lambda\cdot\left(\HSnorm{\gamma}{0}{u}\right)^{1-\lambda}.
\]
 \item [(b)] Let $0\leq s<u \leq 1$. Assume $\gamma\in \Hold{0}{u}{\Omega}{Z}$ and let $\lambda\in ]0,1[$. Then we have
\[
 \Hnorm{\gamma}{0}{\lambda s + (1-\lambda)u} \leq 2\left(\Hnorm{\gamma}{0}{s}\right)^\lambda\cdot\left(\Hnorm{\gamma}{0}{u}\right)^{1-\lambda}.
\]
\end{itemize}
\end{proposition}
\begin{proof}
(a) We may estimate:
 \begin{align*}
  \frac{\Znorm{\gamma(x)-\gamma(y)} }{\quad\quad\quad \Xnorm{x-y}^{\lambda s + (1-\lambda)u}}
						  &	=	\frac{ \Znorm{\gamma(x)-\gamma(y)}^{\lambda}\cdot \Znorm{\gamma(x)-\gamma(y)}^{1-\lambda}  }{\Xnorm{x-y}^{\lambda s}\cdot \Xnorm{x-y}^{(1-\lambda)u}}
						\\&	=	{\underbrace{\left( \frac{\Znorm{\gamma(x)-\gamma(y)} }{\Xnorm{x-y}^s}	\right)}_{\leq \HSnorm{\gamma}{0}{s}}}^\lambda
								\cdot
								{\underbrace{\left( \frac{\Znorm{\gamma(x)-\gamma(y)} }{\Xnorm{x-y}^u}	\right)}_{\leq \HSnorm{\gamma}{0}{u}}}^{1-\lambda}.
 \end{align*}
 This shows (a).

(b)
 \renewcommand{\tt}{\lambda s + (1-\lambda)u}
 \begin{align*}
  \Hnorm{\gamma}{0}{\tt}	  &	=	\supnorm{\gamma} + \HSnorm{\gamma}{0}{\tt}
				\\&	\leq	\supnorm{\gamma}^\lambda\cdot\supnorm{\gamma}^{1-\lambda}
					+	\left(\HSnorm{\gamma}{0}{s}\right)^\lambda\cdot\left(\HSnorm{\gamma}{0}{u}\right)^{1-\lambda}
				\\&	\leq	\Hnorm{\gamma}{0}{s}^\lambda\cdot\Hnorm{\gamma}{0}{u}^{1-\lambda}
					+	\Hnorm{\gamma}{0}{s}^\lambda\cdot\Hnorm{\gamma}{0}{u}^{1-\lambda}.	\qedhere
 \end{align*}
\end{proof}

\begin{proposition}								\label{prop_holder_limit}
 Let $(k,s_0)\in \N_0\times [0,1[$ be given. Then the direct limit space
\[
 \Hold{k}{>s_0}{\Omega}{Z} := \bigcup_{t\in]s_0,1]} \Hold{k}{t}{\Omega}{Z}
\]
 is Hausdorff and compactly regular.
\end{proposition}

\begin{proof}
Since for every $t>s_0$ the inclusion map
\[
 \nnfunc{\Hold{k}{t}{\Omega}{Z} }{\Hold{k}{s_0}{\Omega}{Z}}{\gamma}{\gamma}
\]
is continuous, it follows from the direct limit property that the inclusion map from the direct limit space into the Banach space
\[
 \nnfunc{\Hold{k}{>s_0}{\Omega}{Z} }{\Hold{k}{s_0}{\Omega}{Z}}{\gamma}{\gamma}
\]
is also continuous. Since it is injective, we know that $\Hold{k}{>s}{\Omega}{Z}$ is Hausdorff.

We will show compact regularity using Proposition 1.10 in \cite{MeinArtikel}. Hence, it suffices to show that for every $u>s_0$ there is a $t\in ]s_0,u[$ such that every space 
$\Hold{k}{s}{\Omega}{Z}$ with $s\in]s_0,t]$ induces the same topology on the set $B:= \oBallin{1}{\Hold{k}{u}{\Omega}{Z} }{0}$.

 \renewcommand{\tt}{\lambda s + (1-\lambda)u}
Therefore, let $u>s_0$ be given. We may chose $t\in]s_0,u[$ arbitrarily. Once again, let $s\in]s_0,t[$. Since $t$ lies between $s$ and $u$, we may write $t=\tt$.
Now, we apply Proposition \ref{prop_log_conv}(a) to $\gamma^{(k)}$ and obtain for every $\gamma\in B$
\[
  \HSnorm{\gamma^{(k)}}{0}{t} \leq 	\Big(\HSnorm{\gamma^{(k)}}{0}{s}\Big)^\lambda
					\cdot	\Big(\underbrace{\HSnorm{\gamma^{(k)}}{0}{u}}_{	\leq 1	}	\Big)^{1-\lambda}.
\]
This inequality shows that the identity from $B\!\subseteq\! \Hold{k}{s}{\Omega}{Z}$ to $B\!\subseteq\! \Hold{k}{t}{\Omega}{Z}$ is continuous. Since the continuity of the inverse map is trivial, we have shown that the topologies coincide.
\end{proof}

\section{Lie groups associated to \Holder-continuous functions}								\label{sec_LIEGROUPS}
In the following, let $G$ be an analytic Banach-Lie group over $\K\in\smset{\R,\C}$ with Lie algebra $\g$. 

Like before, $\Omega\subseteq X$ is an open bounded convex subset of a real Banach space $X$ with $\diam\Omega\leq 1$. Let $(k,s)\in\N_0\times[0,1]$ be fixed.
We may define a pointwise Lie bracket on the function space $\Hold{k}{s}{\Omega}{\g}$ and by Theorem \ref{thm_small}, this bracket is continuous with operator norm at most $C_k$. Throughout this section, $C_k$ will always denote these constants introduced in Theorem \ref{thm_small}. Note that they do not depend on the space $\g$.

Now we can compose each $\gamma\in\Hold{k}{s}{\Omega}{\g}$ with the exponential function and obtain the following map:
\[
 \Func{ \Exp_{(k,s)} }{\Hold{k}{s}{\Omega}{\g}}{C(\Omega,G)}{\gamma}{\exp_G\circ \gamma.}
\]

\begin{theorem}[Lie groups associated with \Holder-continuous functions (Banach case)]					\label{thm_LIE_BANACH}
 Let $(k,s)\in\N_0\times[0,1]$ and a Banach-Lie group $G$ with Lie algebra $\g$ be given. Then there exists a unique Banach-Lie group structure on the group 
\[
 \Hold{k}{s}{\Omega}{G}:= \generatedby{\set{\exp_G \circ \gamma}{\gamma\in \Hold{k}{s}{\Omega}{G}}} \leq \CC(\Omega , G)
\]
 such that
\[
 \func{ \Exp_{(k,s)} }{ \Hold{k}{s}{\Omega}{\g}   }{ \Hold{k}{s}{\Omega}{G} }{\gamma}{\exp_G\circ \gamma}.
\]
becomes a local diffeomorphism around $0$.
\end{theorem}
\begin{proof}
We start by choosing a compatible norm $\gnorm{\cdot}$ on $\g$ with the additional property that
\[
 \gnorm{[x,y]}\leq \frac{1}{C_k}\gnorm{x}\gnorm{y}
\]
 for all $x,y\in\g$. This means that  $\opnorm{[\cdot,\cdot]}\leq \frac{1}{C_k}$.
Then the space $\Hold{k}{s}{\Omega}{\g}$ carries a continuous Lie bracket of operator norm at most $1$, due to Theorem \ref{thm_small}:
\[
 \smfunc{[\cdot,\cdot]}{\Hold{k}{s}{\Omega}{\g}\times\Hold{k}{s}{\Omega}{\g}}{\Hold{k}{s}{\Omega}{\g}}
\]
 turning it into a Banach-Lie algebra.
The Lie algebra $\g$ becomes a closed Lie subalgebra of $\Hold{k}{s}{\Omega}{\g}$ by identifying elements of $\g$ with constant functions.

Now, having transferred the Banach-Lie algebra structure from $\g$ to $\Hold{k}{s}{\Omega}{\g}$, we would like to do the same with the group structure.

It is known that (see e.g. \cite[Chapter II, \S 7.2, Proposition 1]{MR1728312}) in a Banach-Lie algebra with compatible norm, the \BCH{}-series converges on 
\[
 U_\g:=\set{(x,y)\in\g\times\g}{\norm{x}+\norm{y}<\log2}
\]
and defines an analytic multiplication: $ \smfunc{*}{U_\g}{\g}.$
Since $\Hold{k}{s}{\Omega}{\g}$ is a Banach-Lie algebra in its own right, we also have a \BCH{}-multiplication there:
$ \smfunc{*}{U_{\Hold{k}{s}{\Omega}{\g}}}{\Hold{k}{s}{\Omega}{\g}}.$
The \BCH{}-series is defined only in terms of iterated Lie brackets. Since addition and Lie bracket of elements in $\Hold{k}{s}{\Omega}{\g}$ correspond to the pointwise operations in $\g$, the \BCH{}-multiplication in $\Hold{k}{s}{\Omega}{\g}$ corresponds to the pointwise \BCH{}-multiplication of functions.

Since $G$ is a Banach-Lie group, it is locally exponential, therefore there is a number $\injepsilon>0$ such that $\exp_G|_{\oBallin{\injepsilon}{\g}{0}}$ is injective. 
Since the \BCH{}-multiplication on $\g$ is continuous, there is a $\delta>0$ such that $\oBallin{\delta}{\g}{0}\times\oBallin{\delta}{\g}{0}\subseteq U_\g$ and $\oBallin{\delta}{\g}{0}*\oBallin{\delta}{\g}{0}\subseteq \oBallin{\injepsilon}{\g}{0}$.

Let $\CC(\Omega , G)$ be the (abstract) group of all continuous maps from $\Omega$ to $G$ with pointwise multiplication.
Then we may define the following map
\[
 \func{\Exp_{(k,s)}}{\Hold{k}{s}{\Omega}{\g}   }{ \CC(\Omega,G) }{\gamma}{\exp_G\circ \gamma}.
\]
The restriction of $\Exp_{(k,s)}$ to $\oBallin{\injepsilon}{\Hold{k}{s}{\Omega}{\g}}{0}$ is injective since $\exp_G|_{\oBallin{\injepsilon}{\g}{0}}$ is injective.

Now, all hypotheses for Corollary 1.8 in \cite{MeinArtikel} are satisfied for \hbox{$U:=\oBallin{\delta}{\Hold{k}{s}{\Omega}{\g} }{0}$,} \hbox{$V:=\oBallin{\injepsilon}{\Hold{k}{s}{\Omega}{\g} }{0}$} and $H:=\CC(\Omega,G)$. Therefore, by Corollary 1.8 in \cite{MeinArtikel}, we get a unique $\CC^\omega$-Lie group structure on the group $\generatedby{\Exp_{(k,s)}( U  )}$ such that 
\[
 \smfunc{\Exp_{(k,s)} |_U}{U\subseteq \Hold{k}{s}{\Omega}{\g} }{ \generatedby{\Exp_{(k,s)}( U  )} }
\]
 is a $\CC^\omega$-diffeomorphism.

But this group, that now has a Lie group structure, is exactly the group $\Hold{k}{s}{\Omega}{G}:= \generatedby{\set{\exp_G \circ \gamma}{\gamma\in \Hold{k}{s}{\Omega}{\g}}}$ defined above. This is the case because for every generator $\exp_G \circ \gamma$ with $\gamma\in \Hold{k}{s}{\Omega}{\g}$ there is an $n\in\N$ such that $\frac{1}{n}\gamma\in U$ and therefore 
\[
 \exp_G \circ \gamma= \exp_G \circ \left(n\cdot \frac{1}{n}\gamma \right)=\left(\exp_G\circ \left(\frac{1}{n}\gamma \right)\right)^n\in \generatedby{\Exp_{(k,s)}(U) }.\qedhere
\]
\end{proof}

\begin{theorem}[Lie groups associated with \Holder-continuous functions ((LB) case)]						\label{thm_LIE_LB}
 Let $(k,s)\in \N_0\times [0,1[$ be given. Then there exists a unique Lie group structure on the group
\[
 \Hold{k}{>s}{\Omega}{G} := \bigcup_{t\in]s,1]} \Hold{k}{t}{\Omega}{G}
\]
 such that
\[
 \Func{\Exp_{(k,>s)}:=\bigcup_{t\in]s,1]}\Exp_{(k,s)}}{\Hold{k}{>s}{\Omega}{\g}}{\Hold{k}{>s}{\Omega}{G}}{\gamma}{\exp_G\circ\gamma}
\]
is a local diffeomorphism around $0$.
\end{theorem}
\begin{proof}
 We wish to use Theorem~C in \cite{MeinArtikel}.
Let $\seqn{t_n}$ be a strictly decreasing cofinal sequence in $]s,1]$, e. g. $t_n:= s + (1-s)\cdot\frac{1}{n}$.
For every $n\in\N$, set \hbox{$G_n := \Hold{k}{t_n}{\Omega}{G}$.} The bonding maps $\smfunc{j_n}{G_n}{G_{n+1} }$ are group homomorphisms.
Since $j_n\circ\Exp_{(k,t_n)}=\Exp_{(k,t_{n+1})}\circ i_n$ with the continuous linear inclusion map $\smfunc{i_n}{\Hold{k}{t_n}{\Omega}{\g} }{ \Hold{k}{t_{n+1}}{\Omega}{\g} }$, we see that each $j_n$ is analytic with $\L(j_n)=i_n$.

Like in the proof of Theorem \ref{thm_LIE_BANACH}, we choose the norm on $\g$ such that
\[
	\gnorm{[x,y]}\leq \frac{1}{C_k}\gnorm{x}\gnorm{y} \hbox{  for $x,y\in\g$. }
\]
Note that this is possible because $k\in\N_0$ is fixed and the $C_k$ do not depend on $s$. This implies that the Lie brackets on the Lie algebras $\Hold{k}{t_n}{\Omega}{\g}$ and the bounded operators 
 $\smfunc{i_n}{\Hold{k}{t_n}{\Omega}{\g}}{\Hold{k}{t_{n+1}}{\Omega}{\g} }$ have operator norm at most $1$.

The locally convex direct limit is Hausdorff by Proposition \ref{prop_holder_limit}, and the exponential map $\Exp=\bigcup_{t\in]s,1]}\Exp_{(k,t)}$
 is injective on the $0$-neighborhood $\bigcup_{t\in]s,1]}\oBallin{\injepsilon}{\Hold{k}{t}{\Omega}{\g} }{0}$.
Hence, by Theorem~C of \cite{MeinArtikel}, there is a unique complex analytic Lie group structure on $G$ such that $\Exp$ is a local diffeomorphism at $0$.
\end{proof}


\addcontentsline{toc}{section}{References}

\bibliography{literatur}{}

\def\polhk#1{\setbox0=\hbox{#1}{\ooalign{\hidewidth
  \lower1.5ex\hbox{`}\hidewidth\crcr\unhbox0}}}
\begin{thebibliography}{1}

\bibitem{MR1728312}
Nicolas Bourbaki.
\newblock {\em Lie groups and {L}ie algebras. {C}hapters 1--3}.
\newblock Elements of Mathematics (Berlin). Springer-Verlag, Berlin, 1998.
\newblock Translated from the French, Reprint of the 1989 English translation.

\bibitem{MeinArtikel}
Rafael Dahmen.
\newblock Analytic mappings between {LB}-spaces and applications in
  infinite-dimensional {L}ie theory, preprint, ar{X}iv: 0807.3655v3 [math.fa].

\bibitem{MR1911979}
Helge Gl{\"o}ckner.
\newblock Infinite-dimensional {L}ie groups without completeness restrictions.
\newblock In {\em Geometry and analysis on finite- and infinite-dimensional
  {L}ie groups ({B}\polhk edlewo, 2000)}, volume~55 of {\em Banach Center
  Publ.}, pages 43--59. Polish Acad. Sci., Warsaw, 2002.

\bibitem{MR830252}
J.~Milnor.
\newblock Remarks on infinite-dimensional {L}ie groups.
\newblock In {\em Relativity, groups and topology, {II} ({L}es {H}ouches,
  1983)}, pages 1007--1057. North-Holland, Amsterdam, 1984.

\end{thebibliography}
\bibliographystyle{plain}

\end{document}